\documentclass{easychair}
\usepackage{palatino}
\usepackage[T1]{fontenc}
\usepackage[utf8]{inputenc}
\usepackage{amsmath}
\usepackage{graphicx}
\usepackage{amssymb}
\usepackage{colortbl} 
\usepackage{xcolor}
\usepackage{fancyhdr}
\usepackage{amsthm}
\usepackage{setspace}
\usepackage{lscape}
\usepackage{mathdots}
\usepackage{amsfonts,hyperref,color}
\usepackage{times}
\usepackage{booktabs}

\theoremstyle{plain}
\newtheorem{thm}{Theorem}
\newtheorem{lem}[thm]{Lemma}
\newtheorem{prop}[thm]{Proposition}
\newtheorem{cor}[thm]{Corollary}

\theoremstyle{definition}
\newtheorem{defn}[thm]{Definition}
\newtheorem{conj}[thm]{Conjecture}
\newtheorem{exmp}[thm]{Example}
\theoremstyle{remark}
\newtheorem{rem}[thm]{Remark}

\title{A formula on the weight distribution of linear codes with applications to AMDS codes}

\author{Alessio Meneghetti\inst{1}\and
Marco Pellegrini\inst{2} \and
Massimiliano Sala\inst{1}}
\institute{Department of Mathematics, University of Trento, Trento, Italy \and
Department of Mathematics and Computer Science, University of Florence, Florence, Italy
}

\authorrunning{Meneghetti, Pellegrini, Sala}

\titlerunning{A formula on the weight distribution}

\begin{document}
\maketitle

\begin{abstract}
The determination of the weight distribution of linear codes has been a fascinating problem since the very beginning of coding theory. There has been a lot of research on weight enumerators of special cases, such as self-dual codes and codes with small Singleton's defect. We propose a new set of linear relations that must be satisfied by the coefficients of the weight distribution.
From these relations we are able to derive known identities (in an easier way) for interesting cases, such as extremal codes, Hermitian codes, MDS and NMDS codes. Moreover, we are able to present for the first time the weight distribution of AMDS codes. We also discuss the link between our results and the Pless equations.
\par\medskip%
\noindent\textbf{AMS subject code classification:} 94B05 Linear codes, general\\
\noindent\textbf{Keywords:} weight distribution, linear codes, AMDS codes, extremal codes
\end{abstract}

\section{Introduction}
The weight distribution of a code $C$ is the vector $A(C)=(A_0,\ldots,A_n)$, where $A_i$ denotes the number of words
with weight $i$ and $n$ denotes the code length.
The determination of the weight distribution of linear codes has been a fascinating problem since the very beginning of coding theory (see \cite{macwilliams1977theory}). Apart from its intrinsic theoretical interest, the knowledge of the weight distribution allows for
the computation of important efficiency paramaters for a code, such as the Probability of Undetected Error (\cite{chang1980simple, wolf1982probability, klove1984probability}) and the Probability of Incorrect Decoding (\cite{huntoon1977computation,faldum2006error, frego2017probability}), at least over common models of channels. The weight distribution can be encoded as coefficients in a polynomial, called the weight enumerator. In this paper, we will write just "distribution" or "enumerator" for brevity.
Since the determination of the distribution implies the determination of the distance, in itself an NP-hard problem \cite{vardy1997intractability}, there is little hope for the existence of a polynomial-time algorithm to solve it for any linear code (starting from its parity-check matrix), and even more so for exact formulae (for the non-linear case see \cite{bellini2018deterministic}).
However, there has been a lot of research in two directions: on one hand, on the connection between the enumerators of related codes (code duality being the classical example, as in \cite{macwilliams1963theorem,pless1963power}), on the other hand, on the establishment of the enumerator of codes having special properties. As regards the latter, in the literature we mainly find results on either codes with symmetries (such as \cite{gleason1970weight, berlekamp1972gleason, macwilliams1972generalizations,mallows1973upper, pless1975classification, harada1997existence, nedeloaia2003weight}) or codes with a small Singleton's defect (\cite{de1996almost, faldum1997codes}). Probably, the most interesting approach of the first type is that of ultraspherical polynomials introduced in \cite{duursma2003extremal}, that led to fast computation of the distribution for extremal divisible codes, while the second produced exact formulae for the distribution of an MDS code (depending only on the code parameters) \cite[Ch. 11, \S3, Theorem 6]{macwilliams1977theory} and of an NMDS (depending also on one weight, e.g. the number of minimum-weight words) in \cite{dodunekov1995near}.
Some classes of codes (whose enumerator is still unknown) have received considerable attention, such as the Hermitian codes (\cite{duursma1999weight, barbero2000weight, marcolla2016small, marcolla2019minimum, marcolla2020hermitian}), where recently, in \cite{pellegrini2019weight}, it has been found out that the computation of their distributions may be significantly less hard than it appeared .

In Section \ref{sec: system} of this paper, we provide formulae (Proposition \ref{prop: number matrices}) that easily determine the complete weight distribution of a code starting from counting  
some special submatrices of $H$ (and viceversa). Although in general this task is as difficult as the original one,
this approach can be fruitful along two directions. The first is to investigate codes having special structure in their parity-check
matrices, the second is to derive a set of linear relations that must be satisfied by the coefficients of the enumerator (Proposition \ref{thm: system}). From these, we derive a further result on linear codes, Proposition \ref{cor: pascal}, where we
show the number of $A_i$’s which need to be known in order to compute (in polynomial time) their full
enumerator.\\
Thanks to Proposition \ref{thm: system}, in Section \ref{sec: amds} we are able to give new proofs for known results from literature regarding MDS codes and
NMDS codes. Moreover, we provide the formula for the weight distribution of AMDS codes (Proposition \ref{thm: formula amds}),
which has been an unsolved problem so far. \\
In Section \ref{sec: comparison}, we use Proposition \ref{thm: system} to deal with some deeply-studied codes such as Hermitian codes and extremal doubly-even self-dual binary codes. while providing a comparison with a special case of Pless’s equations. \\
In Section \ref{sec: pless}, we observe some links between the celebrated Pless equations (here reported as Theorem \ref{thm: pless1}) and our results.
In particular, there is a consequence of the Pless equations (here reported as Theorem \ref{thm: pless}) for which
it would have been possible to obtain formulae for the weight distributions of many families of codes, such as
those provided in this work. Indeed, we prove that, at least in some cases, Proposition \ref{thm: system}  and Theorem \ref{thm: pless} give equivalent relations.
In this section we also provide arguments supporting the convenience of using our results alongiside classical results.\\
Finally, in Section \ref{sec: conclusions} we draw our conclusions and sketch some future research directions.


\section{Preliminaries}\label{sec: preliminaries}
We consider the finite field $\mathbb{F}_q$, with $q$ a prime power $p^m$. An $[n,k]_q$ (linear) code $C$ is a vector subspace of $\left(\mathbb{F}_q\right)^n$, where $n$ is the length of $C$. As usual, we denote with $k$ the dimension of the code, i.e. the dimension of $C$ as vector subspace. The elements of $C$ are known as codewords. Throughout this paper we use the classical notation used in Coding Theory, so any codeword, and in general any vector, is considered as a row vector.
Let $\mathrm{d}(u,v)$ be the Hamming distance between two vectors and $\mathrm{w}(u)$ the Hamming weight of $u$. 
We denote with $d$ the minimum distance of $C$, i.e. the minimum among the Hamming distances between any two distinct codewords. Equivalently, $d$ is the minimum among the Hamming weights of the non-zero codewords of $C$. We say that an $[n,k]_q$ code with minimum distance $d$ is an $[n,k,d]_q$ code. It is well known that $d\leq n-k+1$ (the Singleton bound), and so we define the Singleton defect as $n-k+1-d$, which is always a non-negative integer. The weight distribution of $C$ is the sequence $\{A_i\}_{i=0,\ldots,n}$, where $A_i=\left|\left\{c\in C \;:\;\mathrm{w}(c)=i\right\}\right|$. We say that two $[n,k]_q$ codes are formally equivalent if they have the same weight distribution. 
The support $\mathrm{supp}(c)\subseteq\{1,\ldots,n\}$ is the set of indices of the non-zero coordinates of $c$, and it holds $\left|\mathrm{supp}(c)\right|=\mathrm{w}(c)$.
\\
A generator matrix $G$ of $C$ is a $k\times n$ matrix whose rows form a basis of $C$. The kernel of $G$ is a vector subspace of $\left(\mathbb{F}_q\right)^n$ with dimension $n-k$, hence it is a code itself. This code is known as the dual code of $C$, and we denote it as $C^\perp$. A generator matrix $H$ of $C^{\perp}$ is known as a parity-check matrix of $C$, and it holds that $G\cdot H^t=0$. We denote with $d^{\perp}$ the minimum distance of $C^{\perp}$. Next theorem is a classical result linking the weights of the codewords of $C$ to the parity-check matrix $H$.
\begin{thm}\cite[Theorem 1.4.13]{huffman2010fundamentals}
\label{thm: weights columns}
Let $C$ be a linear code with parity-check matrix $H$. If $c\in C$, the columns
of $H$ corresponding to the nonzero coordinates of $c$ are linearly dependent. Conversely,
if a linear dependence relation with nonzero coefficients exists among $w$ columns of $H$,
then there is a codeword in $C$ of weight $w$ whose nonzero coordinates correspond to these
columns.\end{thm}
In particular, if $c\in C$ then the columns of $H$ identified by $\mathrm{supp}(c)$ are linearly dependent, and if all $(n-k)\times \delta$ submatrices of $H$ have rank equal to $\delta\leq n-k$,  then $d> \delta$.


\section{A linear system for the weight distribution computation}
\label{sec: system}
In this section  and in the rest of the paper we adopt the notation introduced in Section \ref{sec: preliminaries}. 
\begin{defn}
\label{def: N}
Let $M$ be an $s\times t$ matrix with entries in a field, and let $\nu$ be an integer such that $1 \leq \nu \leq t$.
\begin{itemize}
\item We define $N_M(\nu, r)$ as the number of $s \times \nu$ submatrices of $M$ of rank $r$.
\item For any subset $\mathcal{I}\subseteq \{1,\ldots, t\}$ with size $\nu$, $\mathcal{I}=\{i_1,\ldots, i_{\nu}\}$ with $i_1<i_2<\ldots<i_{\nu}$,  we define $M_{[\mathcal{I}]}$ as the $s\times \nu$ submatrix of $M$ identified by the column indices $\mathcal{I}$.  
\end{itemize}
\end{defn}
\begin{prop}\label{prop: number matrices}
Let $C$ be a linear code of length $n$.
Let $\nu$ be an integer
such that $1 \leq \nu \leq n$ and let $H$ be a parity-check matrix of  $C$. Then
\begin{equation}\label{eq: big formula}
\sum_{s=0}^{\nu}\binom{n-s}{\nu-s}A_s= \sum_{r=0}^{\nu} N_H(\nu,r)q^{\nu-r}\;.
\end{equation}

\end{prop}

\begin{proof}
Suppose $r\leq \nu$ is the rank  of $H_{[\mathcal{I}]}$ and $V_{[\mathcal{I}]}$ its kernel. Then, any element $v=(v_1,\ldots,v_{\nu})\in V_{[\mathcal{I}]}$ can be mapped into a codeword of weight less or equal to $\nu$, as follows. Consider the map  
\begin{equation}\nonumber
\varphi_{[\mathcal{I}]}:V_{[\mathcal{I}]}\to\left(\mathbb{F}_q\right)^n,
\qquad
\varphi_{[\mathcal{I}]}(v)=(\bar{v}_1,\ldots,\bar{v}_n)\;,
\end{equation}
such that $\bar{v}_{i_j}=v_j$ for any $j\in\{1,\ldots,\nu\}$, i.e. $\bar{v}_{[\mathcal{I}]}=v$, and $\bar{v}_{l}=0$ if $l\notin\mathcal{I}$.

We call $\varphi_{\nu}$ the map
\begin{equation}\nonumber
\varphi_{\nu}: \; \bigsqcup_{\mathcal{I}:\;|\mathcal{I}|=\nu}V_{[\mathcal{I}]}\;\longrightarrow \; C\;,
\end{equation}
such that $\varphi_{[\mathcal{I}]}$ is the restriction of $\varphi_{\nu}$ to $V_{[\mathcal{I}]}$.\\
We count the elements belonging to the domain of $\varphi_{\nu}$ in two distinct ways:
\begin{itemize}
\item If $H_{[\mathcal{I}]}$ has rank $r$, then $\left|V_{[\mathcal{I}]}\right|=q^{\nu-r}$. Hence, using Definition \ref{def: N}, we have
\begin{equation}\label{eq: proof relation 1}
\left|
\bigsqcup_{\mathcal{I}:\;|\mathcal{I}|=\nu}V_{[\mathcal{I}]}
\right|
=
\sum_{r=0}^{\nu} N_H(\nu,r)q^{\nu-r}\;.
\end{equation}
\item We consider now a codeword $c\in C$ with weight $s\leq \nu$.  Let $\mathcal{I}_1
=\mathrm{supp}(c)$. Any choice of $\nu-s$ indices $\mathcal{I}_2
\subset \{1,\ldots,n\}\smallsetminus
\mathcal{I}_1$ identifies uniquely an element in $\bigsqcup_{\mathcal{I}:\;|\mathcal{I}|=\nu}V_{[\mathcal{I}]}$. More precisely, $\mathcal{I}_1\cup \mathcal{I}_2$ determines uniquely $H_{[\mathcal{I}_1\cup \mathcal{I}_2]}$, clearly $c_{[\mathcal{I}_1\cup \mathcal{I}_2]}\in V_{[\mathcal{I}_1\cup \mathcal{I}_2]}$,  and so there is a unique element $v$ which belongs to $V_{[\mathcal{I}_1\cup \mathcal{I}_2]}$ such that $\varphi_{[\mathcal{I}_1\cup \mathcal{I}_2]}(v)=c$, that is, $v=c_{[\mathcal{I}_1\cup \mathcal{I}_2]}$. To determine the size of $\varphi^{-1}_{\nu}(c)$, the fiber of $c$ under the map $\varphi_{\nu}$, it is therefore enough to count all possible subsets of $ \{1,\ldots,n\}\smallsetminus
\mathcal{I}_1$
 with size $\nu-s$. It follows that the fiber of each codeword of weight $s$ has $\binom{n-\nu}{\nu-s}$ elements, and we observe that all together the fibers of such codewords form a partition of $\bigsqcup_{\mathcal{I}:\;|\mathcal{I}|=\nu}V_{[\mathcal{I}]}$. Since there exist $A_s$ codewords of weight $s$, we obtain
\begin{equation}\label{eq: proof relation 2}
\left|
\bigsqcup_{\mathcal{I}:\;|\mathcal{I}|=\nu}V_{[\mathcal{I}]}
\right|
=
\sum_{s=0}^{\nu} \binom{n-\nu}{\nu-s}A_s\;.
\end{equation}
\end{itemize}
Putting together \eqref{eq: proof relation 1} and \eqref{eq: proof relation 2}, we obtain \eqref{eq: big formula}.
\end{proof}

\newpage
\begin{lem}\label{lem: large w}
If $n - d^{\perp} < \nu \leq n$, then all the $(n - k) \times \nu$ submatrices of $H$ have
rank $n - k$.
\end{lem}
\begin{proof}
Let $d^{\perp} > n - \nu$. If we have an $(n - k) \times \nu$ submatrix $H_{[\mathcal{I}]}$ of $H$ with rank $r < n - k$, then its $n - k$ rows are dependent. By a reordering of its columns, we can suppose that $\mathcal{I}=\{1,\ldots,\nu\}$.
We observe that there is a non-zero vector $v\in \left(\mathbb{F}_q\right)^{n-k}$ such that $v\cdot H_{[\mathcal{I}]}=0$, hence the weight of $v\cdot H$ is at most $n-\nu$. Since $v\cdot H$ is a codeword of $C^{\perp}$ and $\mathrm{w}(v\cdot H)\leq n-\nu$, we have $d^{\perp}\leq n-\nu$, which is a contradiction to our hypothesis $d^{\perp} > n - \nu$.
\end{proof}

\begin{prop}\label{thm: system}
Let $\left\{A_i\right\}$ be the weight distribution of $C$. Let $\nu$ be an integer such that $n-d^{\perp}<\nu\leq n$. Then
\begin{equation}\label{eq: thm system}
\sum_{s=0}^\nu
\binom{n-s}{\nu-s}A_s
 =
\binom{n}{\nu}q^{\nu+k-n}\;.
\end{equation}
\end{prop}
\begin{proof}
Lemma \ref{lem: large w} implies that 
\begin{itemize}
\item $N_H(\nu,r)=0$ for $\nu >n-d^{\perp}$ and $r\neq n-k$, and
\item $N_H(\nu,r)=\binom{n}{\nu}$ otherwise.
\end{itemize}
Then, we can substitute these values into Equation \eqref{eq: big formula}, proving our claim.
\end{proof}
\begin{prop}\label{cor: pascal}
Let $\sigma$ be the sum of the Singleton defects of $C$ and $C^{\perp}$. 
The knowledge of $\sigma+d-1$ elements of the weight distribution $\{A_0,\ldots,A_n\}$ is enough to compute the full weight distribution of $C$ and $C^{\perp}$.
This computation can be achieved by a direct application of \eqref{eq: thm system}.
\\
In particular, the knowledge of $d$ and of any $\sigma-1$ elements of $\{A_d,\ldots,A_n\}$ is enough to compute the entire weight distribution of $C$ and $C^{\perp}$.
\end{prop}
\begin{proof}
The linear system obtained by considering Equation \eqref{eq: thm system} for each $\nu$ in the range $\{n-d^{\perp}+1,\ldots, n\}$ can be written as 
\begin{equation}\label{eq: pascal}
\mathcal{P}_{d^{\perp},n}\cdot A(C)=b_{d^{\perp}} ,
\end{equation}
where
\begin{itemize}
\item $
\mathcal{P}_{r,t}=\begin{bmatrix}\binom{t-j}{i}\end{bmatrix}_{i=0,\ldots,r-1, \;j=0,\ldots,t}
$, 
with $\binom{t-j}{i}=0$ whenever $i+j> t$,
\item $A(C)=(A_0,\ldots, A_n)^t$ is a column vector of length $n+1$ containing the weight distribution of $C$, and
\item $b_{r}=\left( \binom{n}{0}q^{k},\binom{n}{1}q^{k-1}, \ldots, \binom{n}{r-1}q^{k-r+1} \right)^{\mathrm{T}}$.
\end{itemize}
Observe that $\mathcal{P}_{r,t}$ is a truncated Pascal matrix with $r\ge 1$ rows and $t+1\ge r$ columns. Hence, as proved in \cite{kersey2016invertibility}, all its minors of order $r$ are non-zero. This implies that by knowing at least $n-d^{\perp}+1$ values of the weight distribution and substituting them into \eqref{eq: pascal} we obtain a linear system with $d^{\perp}$ equations in $d^{\perp}$ unknowns which admits a unique solution. Furthermore, the knowledge of $d$ implies the knowledge of $A_0,\ldots,A_{d-1}$, hence it is enough to know other $n-d^{\perp}+1-d=\sigma-1$ $A_i$'s.
\end{proof}
The previous proposition is optimal, in the following sense.
Let $\mathrm{C}$ be the set of all linear codes with non-trivial parameters, that is, $q\geq 2$, $n>k>0$.
We define the function $h: \mathrm{C} \to \mathbb{N}$, where for any $C \in \mathrm{C}$ the value $h(C)$ is the smallest number such that if any $h(C)$ elements of
the weight distribution of $C$ are fixed, say $A_{i_1},\ldots,A_{i_h}$, then the rest of the weight distribution can be deterministically deduced with the sole knowledge of both the code parameters ($n,k,d,d^{\perp}$)
and the weights $A_{i_1},\ldots,A_{i_h}$.
We now consider all functions $\zeta_s:  \mathrm{C}  \to \mathbb{N}$ such that:
\begin{itemize}
\item for any code $C \in \mathrm{C}$, $h(C) \leq \zeta_s(C)$; 
\item there is an integer $s$ for which $\zeta_s(C)=n-(d+d^{\perp})+s$.
\end{itemize}
The set formed by all such $\zeta_s$ functions is non-empty, since for example our previous proposition gives $\zeta_1$.\\
We now claim that there is no better $\zeta_s$. In other words, there is no $s\leq 0$ for which $\zeta_s$ exists.
To show this, let us consider two near-MDS codes with the same $q,n,k,d,d^{\perp}$. If such a $\zeta_s$ exists, then only $n-(d+d^{\perp})=0$ weights would be necessary to determine the others (instead of $\sigma-1=n-(d+d^{\perp})+1$).
Which means that the two near-MDS codes would be formally equivalent. Generally speaking this is false, as shown by taking the two codes in the following example.
\begin{exmp}
Let $C_1$ and $C_2$ be the $[8, 4, 4]$ codes over $\mathbb{F}_4=\{0,1,\alpha,\alpha^2\}$ generated respectively by $G_1$ and $G_2$ with
$$
G_1=\begin{bmatrix}
1&   0&   0&   0&   1& \alpha^2&   \alpha&   0\\
0&   1&   0&   0&   0&   1& \alpha^2&   \alpha\\
0&   0&   1&   0&   \alpha&   \alpha&   0&   1\\
0 &  0&   0&   1&   1& \alpha^2&   1&   1
\end{bmatrix}\;,
\qquad
G_2=\begin{bmatrix}
  1&   0&   0&   0&   1& \alpha^2&   \alpha&   0\\
  0&   1&   0&   0& \alpha^2&   0&   1&   1\\
  0&   0&   1&   0&   0&   1& \alpha^2& \alpha^2\\
  0&   0&   0&   1&   \alpha&   \alpha&   0&   1
\end{bmatrix}\;.
$$
The weight distributions of $C_1$ and $C_2$ are respectively
$$
\{1,0,0,0,27,60,78,60,30\}
\quad\mathrm{and}\quad
\{1,0,0,0,30,48,96,48,33\}\;.
$$
\end{exmp}
We actually have a much stronger conjecture, as follows. 
\begin{conj}\label{conj: optimality}
Let us consider any function $\zeta: \mathrm{C}  \to \mathbb{N}$ such that for any code $C$ in $\mathrm{C}$, $h(C) \leq \zeta(C)$.
Then 
$$
     n-(d+d^{\perp})+1 \leq \zeta\;.
$$
\end{conj}
\noindent Observe that we are dropping here any requirement on the shape of $\zeta$, which might be nonlinear.

\section{MDS, near-MDS and almost-MDS codes}
\label{sec: amds}
\begin{thm}\label{thm: mds}
The weight distribution of MDS codes is given by
$$
A_w=\binom{n}{w}\sum_{j=0}^{w-d}(-1)^j\binom{w}{j}\left(q^{w-d+1-j}-1\right)
$$
for each $w$ in $\{d,\ldots,n\}$ .
\end{thm}
\noindent This well-known formula (see e.g. \cite[Ch. 11, \S3, Theorem 6]{macwilliams1977theory} or \cite{pless1998introduction}) can be proved directly using Proposition \ref{thm: system} and Proposition \ref{cor: pascal}. Indeed, if $C$ is an MDS code, so is its dual, and the Singleton defect of an MDS code is by definition equal to $0$; it follows that $\sigma=0$, hence by Proposition \ref{cor: pascal}  the weight distribution of MDS codes can be directly obtained from Proposition \ref{thm: system} only by the knowledge of the length $n$ and minimum distance $d$ of the code. For example, in the case $\nu=d$, Equation \eqref{eq: thm system} is
$$
\binom{n}{d}A_0+\binom{n-d}{0}A_d=\binom{n}{d}q^{d+k-n}\;,
$$  
which becomes 
$$
A_d=\binom{n}{d}\left(q-1\right).
$$

Almost-MDS (AMDS) codes are defined as codes with Singleton's defect $1$, namely, almost-MDS codes are $[n,k,n-k]_q$ codes. A sub-class of AMDS codes are the so-called near-MDS (NMDS) codes, defined as AMDS codes whose dual is still an AMDS code. In this particular case the sum of the Singleton defects of a code and its dual is $\sigma=2$, hence the knowledge of a single element among $A_d,\ldots,A_n$ completely determines the entire weight distribution.
\begin{thm}\label{thm: near mds}
Let $C$ be an $[n,k,n-k]_q$ near-MDS code, i.e. a code with Singleton's defect $1$ and $\sigma=2$. \\
Let $1\leq i \leq k$. Then
$$
A_{n-k+i}=\binom{n}{k-i}\sum_{j=0}^{i-1}(-1)^j\binom{n-k+i}{j}(q^{i-j}-1)+(-1)^i\binom{k}{i}A_{n-k}\; .
$$
\end{thm}
A first proof of this result can be found in \cite{dodunekov1995near}. We remark that, similarly to Theorem \ref{thm: mds}, it is possible to prove the formula for the distribution of NMDS codes by using Proposition \ref{thm: system}. 

\par\medskip

The weight distribution of AMDS codes depends on more than one parameter, and a general formula is not known. In the remaining part of this section we make use of Proposition \ref{thm: system} and Proposition \ref{cor: pascal} to obtain a explicit formula for AMDS codes, i.e. for $[n,k,n-k]_q$ codes whose dual are $[n,n-k,k-\sigma+2]_q$ codes, with $\sigma \ge 2$. Notice that NMDS are particular cases of AMDS codes, corresponding to $\sigma=2$.

When $d=n-k$ and $d^{\perp}=k-\sigma+2$, by writing $\nu$ as $n-k+\sigma-1+i$ with $i$ in the range $0,\ldots,k-\sigma+1$,  Equation \eqref{eq: thm system} becomes
\begin{equation}\nonumber
\sum_{s=0}^{n-k+\sigma-1+i}
\binom{n-s}{n-k+\sigma-1+i-s}A_s
 =
\binom{n}{n-k+\sigma-1+i}q^{i+\sigma-1}\;.
\end{equation}
Since we know that
\begin{equation}\nonumber
A_0=1,\qquad
A_s=0\;\mathrm{for}\;1\leq s\leq n-k-1\;,
\end{equation}
then we obtain
\begin{multline*}
\binom{n}{n-k+\sigma-1+i}+
\sum_{h=0}^{\sigma-2}\binom{k-h}{\sigma-1+i-h}A_{n-k+h}+
\\+
\sum_{s=n-k+\sigma-1}^{n-k+\sigma-1+i}
\binom{n-s}{n-k+\sigma-1+i-s}A_s
 =
\binom{n}{n-k+\sigma-1+i}q^{i+\sigma-1}\;,
\end{multline*}
hence, by denoting $s=n-k+\sigma-1+j$, we obtain the formula
\begin{multline}\label{eq: re-system}
\sum_{j=0}^{i}
\binom{k-\sigma+1-j}{i-j}A_{n-k+\sigma-1+j}=
\\
 =
\binom{n}{n-k+\sigma-1+i}\left(q^{i+\sigma-1}-1\right)
-
\sum_{h=0}^{\sigma-2}\binom{k-h}{\sigma-1+i-h}A_{n-k+h}\;.
\end{multline}
Observe that we can write the last equation in matrix form as $\mathcal{P}\cdot A=b$, thus $A=\mathcal{P}^{-1}\cdot b$, where the matrix $\mathcal{P}$ is the Pascal matrix $\left[
\binom{k-\sigma+1-j}{i-j}
\right]_{i,j=0,\ldots,k-\sigma+1}$, $A$ is the column vector $(A_{n-k+\sigma-1},\ldots, A_n)^{\mathrm{T}}$, and $b$ is the column vector of known terms on the right-hand side of equation \eqref{eq: re-system}. 
Pascal matrices and their properties are deeply studied (see e.g. \cite{brawer1992linear, yang2006explicit,kersey2016invertibility}), and in our case we have
$$
\mathcal{P}^{-1}=
\left[
(-1)^{i-j}
\binom{k-\sigma+1-j}{i-j}
\right]_{i,j=0,\ldots,k-\sigma+1}\;.
$$
By expliciting the multiplication $\mathcal{P}^{-1}\cdot b$ we obtain the formula in the following proposition.
\begin{prop}\label{thm: formula amds}
Let $C$ be an $[n,k,n-k]_q$ AMDS code, let $C^{\perp}$ be an $[n,n-k,k-\sigma+2]_q$ code, and let $\{A_0,\ldots, A_n\}$ be the weight distribution of $C$. Then, the knowledge of $\{A_{n-k}, \ldots, A_{n-k+\sigma-2}\}$ is enough to compute the entire weight distribution of $C$. In particular, for any $0\leq i \leq k-\sigma+1$, $A_{n-k+\sigma-1+i}$ is
{\small\begin{multline*}
\sum_{j=0}^{i}(-1)^{i-j}
\binom{k-\sigma+1-j}{i-j}
\left[
\binom{n}{n-k+\sigma-1+j}\left(q^{j+\sigma-1}-1\right)
-
\sum_{h=0}^{\sigma-2}\binom{k-h}{\sigma-1+j-h}A_{n-k+h}
\right]\;.
\end{multline*}}
\end{prop}

\par\bigskip

\section{Comparison with known formulae}
\label{sec: comparison}
In this section we show how some known results on the weight distribution of codes can be derived as particular cases of Propositions \ref{thm: system} and \ref{cor: pascal}.

\par\medskip

\subsection{Hermitian codes}
\label{subsec: hermitian}
The Hermitian curve $\mathcal{H}$ over $\mathbb{F}_{q^2}$ is defined by the equation $x^{q+1}=y^q+y$ and has $n = q^3$ rational affine points $P_1,\ldots, P_n$, plus one point at infinity $P_{\infty}$.
The Hermitian code is defined by the general Goppa construction:
\begin{defn}
Let $m$ be a non-negative integer, $D=\sum_{i=1}^nP_i$, $Q=mP_{\infty}$, and let $L(Q)$ be the vector space of rational functions on $\mathcal{H}$ whose divisor of poles is bounded by $Q$. The Hermitian code is
$$
\mathcal{C}(q,m)=\{(f(P_1),\ldots,f(P_n))\in\left(\mathbb{F}_{q^2}\right)^n\mid f\in L(Q)\}\;.
$$
\end{defn}
We remark that for each $1\leq k\leq n-1$ there exists a (unique)  Hermitian code with dimension $k$. To ease the notation we simply denote such code with $\mathcal{C}_k$ and with $\mathcal{A}=\{\mathcal{A}_i\}_{i=0,\ldots,n}$ its weight distribution.

In \cite{pellegrini2019weight}, the authors prove that there exist linear relations among the weights of Hermitian codes. Moreover, they prove that the knowledge of $\mathcal{A}_d,\ldots, \mathcal{A}_{n-d^{\perp}}$ is enough to determine the entire weight distribution. This particular subset of $n-d-d^{\perp}+1$ elements of $\mathcal{A}$ is called the \textit{critical region}, as in the following definition. 
\begin{defn}[\cite{pellegrini2019weight}, Definition 11]
Let $q$ be a power of a prime. For any $k$ such that $1\leq k\leq n-1$, let $d=d(k)$ be the distance of the Hermitian code $\mathcal{C}_k$ with first parameter $q$ and dimension $k$, and let $d^{\perp}=d^{\perp}(k)$ the distance of $\mathcal{C}_k^{\perp}$. The critical region $R_q$ is the set of pairs $(k,w)$ such that $d\leq w\leq n-d^{\perp}$:
$$R_q:=\{(k, w)\in \mathbb{N}^2\;:\;d(k)\leq w\leq n-d^{\perp}(k)\}\;.$$
\end{defn}
The main result in \cite{pellegrini2019weight} is the following theorem:
\begin{thm}[\cite{pellegrini2019weight}, Theorem 15]\label{thm: hermitian}
Let us consider the Hermitian code $\mathcal{C}_k$. Let $d^{\perp}$ be the distance of $\mathcal{C}_k^{\perp}$. Let $\nu$ be an integer such that $n-d^{\perp}<\nu\leq n$. Then 
$$
\sum_{s=0}^\nu \binom{q^3-s}{\nu-s}\mathcal{A}_s=\binom{q^3}{\nu}q^{2(\nu+k-q^3)}\;.
$$
The knowledge of $\mathcal{A}_w$ with $d\leq w\leq n-d^{\perp}$, that is, the subset of the weight distribution that lies in the critical region, allows us to calculate the whole weight distribution of $\mathcal{C}_k$.
\end{thm}
Theorem \ref{thm: hermitian} can be seen as a particular case of Proposition \ref{thm: system} specialized to Hermitian codes. Moreover, the critical region can be generalized to any pattern of $n-d-d^{\perp}+1$ elements of the weight distribution, as stated in Proposition \ref{cor: pascal}, but we do not write this obvious generalization.

\subsection{Extremal doubly-even self-dual binary code}
\label{subsec: extremal}
Since the seminal work of Gleason in 1970 \cite{gleason1970weight} much effort has been put on the characterisation of extremal codes, in paticular of type II codes, namely extremal doubly-even self-dual binary code. Remarkable results were produced in 1972 by Mallows and Sloane, who  showed an explicit formula for the distribution \cite[Theorem 1]{mallows1973upper}. Another milestone is the work of Duursma on ultraspherical polynomials published in 2003, in which implicit linear relations between the coefficients of the distribution are obtained \cite[Theorem 12]{duursma2003extremal}, leading to fast methods for the distribution computation. 

Extremal type II codes are $[24m,12m,4m+4]_2$ self-dual codes (hence both $d$ and $d^{\perp}$ are equal to $m+4$). Their weight distribution has two properties:
\begin{itemize}
\item $A_i$ is zero for each $i$ not divisible by 4;
\item $A_i=A_{n-i}$, hence $A_{24m}=1$ and $A_i=A_{n-i}=0$ for each $1\leq i \leq 4m+3$.
\end{itemize}
By Proposition \ref{cor: pascal}, to completely determine the distribution we need to know $d$, $d^{\perp}$ and $16m-7$ elements in $\{A_{4m+4},\ldots, A_{24m}\}$. The number of non-zero elements in this set is $4m$ (including $A_n=1$), hence we know $16m-3$ elements (i.e. the zero elements in the set). Since $16m-3$ is larger than $16m-7$, then we know enough elements to solve the linear system in Proposition \ref{thm: system}.
\\
Let $a,b\in \mathbb{R}$. $\delta(a,b)$ is the Kronecker delta, that is $\delta(a,a)=1$ and $\delta(a,b)=0$ if $a\neq b$.
\begin{prop}\label{prop: extremal}
The weight distribution of an $[24m,12m,4m+4]_2$ extremal doubly-even self-dual binary code is completely determined by any $4m-1$ relations of the form
\begin{equation}\label{eq: our extremal}
\sum_{\ell=1}^{4m-1}\binom{20m-4\ell}{\nu-4m-4\ell}A_{4m+4\ell}=\binom{24m}{\nu}\left(2^{\nu-12m}-1\right)-\delta(24m,\nu)\;,
\end{equation}
where $20m-4<\nu\leq 24m$.
\end{prop}
\begin{proof}
We specialize the linear system in Proposition \ref{thm: system} to our case, i.e. $n=24m$, $k=12m$, $d=d^{\perp}=4m+4$, $ A_i=A_{24m-i}$, and $A_i=0$ for each $i\neq 0\mod4$. We remark that the symmetry of the weight distribution implies $A_{24m}=1$ and $A_{24m-i}=0$ for each $1\leq i\leq 4m-3$. The term $\delta(24m,\nu)$ appears since $A_{24m}=1$. 
\\
The linear system admits a unique solution by Proposition \ref{cor: pascal}, which also implies that any subset of $4m-1$ equations allows to determine the entire distribution. 
\end{proof}
We remark that to speed-up the computation it is possible to add to the equations in Proposition \ref{prop: extremal} the $2m-1$ symmetry equations $A_{4m+4\ell}=A_{20m-4\ell}$, with $1\leq\ell\leq 2m-1$. Even more, it is possible to consider a subset of $2m$ equations from Proposition \ref{prop: extremal} and to use all symmetry equations to obtain a full-rank system with $4m-1$ unknowns and an equal number of equations, whose unique solution is therefore the weight distribution.  It is however important to choose wisely which equations \eqref{eq: our extremal} are to be kept. For example, in the case $m=1$ we choose $2m=2$ equations from \eqref{eq: our extremal} and consider the symmetry equation $A_8-A_{16}=0$. If we choose the two equations from \eqref{eq: our extremal} corresponding to $\nu=22$ and $\nu=24$ we obtain the system
\begin{equation}\label{eq: exmp extr 1}
\left\{
\begin{array}{l}
\binom{16}{14}A_8+\binom{12}{10}A_{12}+\binom{8}{6}A_{16}=\binom{24}{22}\left(2^{10}-1\right)\\
\binom{16}{16}A_8+\binom{12}{12}A_{12}+\binom{8}{8}A_{16}=\binom{24}{24}\left(2^{12}-1\right)-1\\
A_8-A_{16}=0\;,
\end{array}
\right.
\end{equation}
while, if we use $\nu=23$ and $\nu=24$ we obtain
\begin{equation}\label{eq: exmp extr 2}
\left\{
\begin{array}{l}
\binom{16}{15}A_8+\binom{12}{11}A_{12}+\binom{8}{7}A_{16}=\binom{24}{23}\left(2^{11}-1\right)\\
\binom{16}{16}A_8+\binom{12}{12}A_{12}+\binom{8}{8}A_{16}=\binom{24}{24}\left(2^{12}-1\right)-1\\
A_8-A_{16}=0\;.
\end{array}
\right.
\end{equation}
In \eqref{eq: exmp extr 1} the three equations are linearly independent, hence the system admits as unique solution the weight distribution of the $[24,12,8]_2$ extremal doubly-even self-dual code.\\
Instead, in \eqref{eq: exmp extr 2} the equations are linearly dependent, hence the weight distribution is not completely determined.

\section{Pless Equations}
\label{sec: pless}
In this section we relate our work with a particular formulation of the Pless equations, which is enunciated in the following theorem.
\begin{thm}\cite[Theorem 7.2.3]{huffman2010fundamentals}\label{thm: pless1}
Let $C$ be a linear code of length $n$, $C^{\perp}$ its dual code, and let $\{A_i\}$ and $\{B_j\}$ be the weight distributions of $C$ and $C^{\perp}$. For any $0\leq \nu\leq n$ it holds
\begin{equation}\label{eq: pless1}
\sum_{i=\nu}^n\binom{i}{\nu}A_i=q^{k-\nu}\sum_{j=0}^{\nu}(-1)^j\binom{n-j}{n-\nu}(q-1)^{\nu-j}B_j \;.
\end{equation}
\end{thm}
If we consider the case $\nu<d^{\perp}$ then each $B_j$ on the right-hand side of Equation \eqref{eq: pless} is equal to zero, with the exception of $B_0$. We obtain the following theorem.
\begin{thm}\label{thm: pless}
For any $\nu< d^{\perp}$ it holds
\begin{equation}\label{eq: pless}
\sum_{i=\nu}^n\binom{i}{\nu}A_i=q^{k-\nu}\binom{n}{\nu}(q-1)^{\nu} \;.
\end{equation}
\end{thm}
Next corollary shows that, in a loose sense, linear system \eqref{eq: thm system} and linear system \eqref{eq: pless} are equivalent, at least when enough terms of the weight distribution are known.
\begin{cor}\label{cor: pless equivalent}
Let $\sigma$ be the sum of the Singleton defects of $C$ and $C^{\perp}$. 
The knowledge of $\sigma+d-1$ elements of the weight distribution $\{A_0,\ldots,A_n\}$ is enough to compute the full weight distribution of $C$ and $C^{\perp}$.
This computation can be achieved by a direct application of \eqref{eq: pless}.
\\
In particular, the knowledge of $d$ and of any $\sigma-1$ elements of $\{A_d,\ldots,A_n\}$ is enough to compute the entire weight distribution of $C$ and $C^{\perp}$.
\end{cor}
\begin{proof}
As in the proof of Proposition \ref{cor: pascal}, the matrix associated to the linear system in Theorem \ref{thm: pless} is a version of a truncated Pascal matrix $\mathcal{Q}_{d^{\perp},n}$, hence all $d^{\perp}\times d^{\perp}$ minors are non-zero (see \cite{kersey2016invertibility}). Therefore, the knowledge of $n-d^{\perp}+1$ elements of $A(C)$ is enough to compute the entire weight distribution.
The linear system in Equation \eqref{eq: pless} has the same number of equations and unknowns as the system in Equation \eqref{eq: thm system}, and both determine the entire weight distribution, provided that $n-d^{\perp}+1$ values among $A_0,\ldots, A_n$ are known.
\end{proof}
Note that from the proof of Corollary  \ref{cor: pless equivalent} it follows that the knowledge of any subset of $n-d^{\perp}+1$ elements of the distribution is enough to retrieve the entire distribution (for example by using Theorem \ref{thm: pless}), despite in many previous results it was believed that the known weights had to be consecutive (see \cite[Theorem 7.3.1]{huffman2010fundamentals}). 
\\
Our proof is based on the uniqueness of the weight distribution, hence the two linear systems are equivalent provided the existence of the investigated code $C$. 
In some cases, by comparing the results obtained by using both Proposition \ref{thm: system} and Corollary \ref{cor: pless equivalent} it could be possible to prove the non-existence of $C$ (i.e. the weight distributions obtained from \eqref{eq: thm system} and \eqref{eq: pless} may be different). 
However, computational evidence suggests that the equations in \eqref{eq: thm system} may be linearly dependent on equations  \eqref{eq: pless} and vice-versa. At the present time, it is therefore unclear whether the two statements can be directly derived one from the other. We finally remark that, due to Corollary \ref{cor: pless equivalent}, similar results to those in Sections \ref{sec: amds}, \ref{subsec: hermitian} and \ref{subsec: extremal} can be obtained by using the Pless equations instead of Proposition \ref{thm: system}.
\begin{rem}
Although \eqref{eq: thm system} and \eqref{eq: pless} are likely to be equivalent, Proposition \ref{prop: number matrices} and Theorem \ref{thm: pless1} are likely to be inequivalent.
In particular, Proposition \ref{prop: number matrices} gives constraints on the weight distribution of a code starting from the submatrices of $H$ with a given rank, rather than from the weights of its dual. Therefore, we think that for codes with highly-structured parity-check matrices, such as evaluation codes over curves, we can deduce directly some significant information on these submatrices and henceforth on their weight distribution.
\end{rem}


\section{Conclusions}
\label{sec: conclusions}
By studying some properties of the minors of parity-check matrices we prove a new set of linear relations between the weights of linear codes. Our main results, Propositions \ref{thm: system} and  \ref{cor: pascal}, imply that the weight distribution of a linear code is completely determined by a number of parameters bounded by the sum of the Singleton defects of the code itself and of its dual.

Proposition \ref{cor: pascal}, together with the simple structure of our linear system (which relies on Pascal matrices), allows a straightforward reformulation of known results, such as formulae for the particular case of Hermitian codes, extremal codes, MDS and near-MDS codes. We are also able to determine for the first time explicit formulae for the weight distribution of any almost-MDS code, where the number of formula parameters depends  only on the dual distance of the code. 

These results open new directions for applications to several families of codes, as hinted by the analysis on Hermitian codes, AMDS codes and extremal codes. Regarding the latter, we aim at extending our formulae to related classes of codes, such as near-extremal formally self-dual codes \cite{kim2007note, han2009nonexistence}, as well as to codes derived from these classes (e.g. punctured or shortened extremal codes). Another direction is to investigate codes having special structure in their parity-check
matrices. Moreover, we are interested in proving Conjecture \ref{conj: optimality}, where we state our belief that Proposition \ref{cor: pascal} is optimal (for a general linear code the weight distribution depends exactly on $n-d-d^{\perp}+1$ parameters).

\section*{Acknowledgements}
This project has been carried on within the EU-ESF activities, call "PON Ricerca e Innovazione
2014-2020", project “Distributed Ledgers for Secure Open Communities”.
\\
Preliminary results of this paper were presented on a talk given at Algebra for Cryptography 2019 in L'Aquila \cite{meneghetti2019A4C}.
\\
The authors would like to thank Jon-Lark Kim for fruitful discussions and comments.

\bibliographystyle{amsalpha}
\bibliography{Refs}
\end{document}